\documentclass[oneside]{amsart}

\usepackage{amsmath,amssymb,amsfonts,amsthm,color}
\usepackage{eucal}
\usepackage{ stmaryrd }
\usepackage{tikz}
\usepackage{verbatim}
\usepackage{epstopdf}

\usetikzlibrary{matrix, arrows}

\usepackage{graphicx,psfrag}

\usepackage{multicol}

\usepackage{ifthen} 

\newcommand{\baseenvskip}{\baselineskip 6.53mm}

\newtheorem{thm}{Theorem}

\newtheorem{question}{Question}

\newtheorem{lemma}[thm]{Lemma}

\newtheorem{claim}[thm]{Claim}

\newtheorem{prop}[thm]{Proposition}

\newtheorem{notation}[thm]{Notation}

\theoremstyle{remark}
\newtheorem{rmk}{Remark}[section]
\newenvironment{remark}{\begin{rmk}\rm\baseenvskip}{\end{rmk}}

\theoremstyle{remark}
\newtheorem{ex}{\rm{\textbf{Example}}}[section]

\newtheorem{dfn}{\rm{\textbf{Definition}}}[section]
\numberwithin{equation}{section} \numberwithin{thm}{section}
\numberwithin{rmk}{section} \numberwithin{figure}{section} \numberwithin{dfn}{section}
\numberwithin{ex}{section}

\newcommand{\ds}{\displaystyle}

\newcommand{\h}{\mathcal{H}}

\newcommand{\R}{\mathbb{R}}

\newcommand{\C}{\mathbb{C}}

\newcommand{\bd}{\partial}
\newcommand{\td}{\tilde}

\newcommand{\mbb}{\mathbb}

\newcommand{\spt}{\text{\textnormal{spt}}}

\newcommand{\llb}{\llbracket}
\newcommand{\rrb}{\rrbracket}
\newcommand{\cross}{\times}
\newcommand{\lip}{\text{\textnormal{Lip}}}
\newcommand{\mass}{\mbb{M}}

\newcommand{\D}{\mbb D}
\newcommand{\Ha}{\mathcal{H}}
\newcommand{\dir}{\text{\textnormal{Dir}}}
\newcommand{\Dir}{\text{\textnormal{Dir}}}

\newcommand{\atantwo}{\text{\textnormal{atan2}}}
\newcommand{\restr} {
\hskip2.5pt{\vrule height7pt width.5pt depth0pt}
\hskip-.2pt\vbox{\hrule height.5pt width7pt depth0pt}
\, }

  \author{Quentin Funk }
       \address{Department of Mathematics, Rice University\\ 6100 Main, MS136 \\
Houston TX 77251-1892}
       \email{qfunk@rice.edu}
         \author{Robert Hardt }
       \address{Department of Mathematics, Rice University\\ 6100 Main, MS136 \\
Houston TX 77251-1892}
       \email{hardt@rice.edu}
 \thanks{The first author was partially supported by NSF DMS-1148609, the second author was partially supported by NSF DMS-1207702.}

\begin{document}

\title[A Multiple-Valued Plateau Problem]{A Multiple-Valued Plateau Problem}

\maketitle

\begin{abstract}
The existence of Dirichlet minimizing multiple-valued functions for given boundary data has been known since pioneering work of F. Almgren. Here we prove a multiple-valued analogue of the classical Plateau problem of the existence of area-minimizing mappings of the disk. Specifically, we find, for $K \in \mathbb N,$ $k_1,...,k_K\in \mathbb N$ with sum $Q$ and any collection of $K$ disjoint Lipschitz Neighborhood Retract Jordan curves, optimal multiple-valued boundary data with these multiplicities which extends to a Dirichlet minimizing $Q$-valued function with minimal Dirichlet energy among all possible monotone parameterizations of the boundary curves. Under a condition analogous to the Douglas condition for minimizers from planar domains, conformality of the minimizer follows from topological methods and some complex analysis. Finally, we analyze two particular cases: in contrast to single-valued Douglas solutions, we first give a class of examples for which our multiple-valued Plateau solution has branch points. Second, we give examples of a degenerate behavior, illustrating the weakness of the multiple-valued maximum principle and provide motivation for our analogous Douglas condition.
\end{abstract}

\section{Introduction}

Multiple-valued functions were first introduced by F. Almgren in \cite{bigreg} to analyze the regularity of mass-minimizing rectifiable currents in arbitrary codimension. Such functions have recently seen a resurgence in the works of C. De Lellis and E. N. Spadaro \cite{DelSeries} (and its continuations), \cite{revisit, Delint}, S. Chang \cite{Schang}, W. Zhu \cite{WZhu2dim, reducing}, C. C. Lin \cite{cclin1}, P. Bouafia \cite{selection}, P. Bouafia and T. De Pauw \cite{Bouafia}, and P. Mattila \cite{Mattila}, to name just a few.

A $Q$-valued function to $R^n$ is essentially a single-valued function taking values in the set $A_Q(R^n)$ of all unordered $Q$ tuples of points in $\R^n$.  Equivalently $A_Q(\R^n)$ may be considered as the space of all sums $\sum_{i=1}^Q [[ a_i ]] $ of $Q$ Dirac measures of points in $R^n$.  Thus, for $\Omega\subset \R^m$, a function $f:\Omega\to A_Q(\R^n)$ corresponds to a $Q$- valued function from $\Omega$ to $\R^n$. We will occasionally refer to such functions as     \emph{multiple-valued functions}.

One can metrize $A_Q(\R^n)$ via a standard `translation' invariant metric. With this metric, Almgren showed that $A_Q(\R^n)$ admits a bi-Lipschitz embedding onto a Lipschitz retract of a high dimensional Euclidean space, and used this to define Sobolev multiple-valued maps. Further, he gave a well defined notion of differentiability and Dirichlet energy, and solved the Dirichlet problem for such functions.

In this paper we seek to formulate and prove a result analogous to the classical Plateau problem in the setting of multiple-valued functions. Our main result is the following:

\begin{thm}
\label{thm:main}
 Let $\Gamma_1, \ \Gamma_2, \ ...\ , \Gamma_K$ be disjoint, simple closed curves of finite length in $\R^n$ so that $\Gamma_i \cap \Gamma_K = \phi$ for $i\neq j$ and $k_1,...,k_K$ positive integers so that $\sum k_i = Q$. Further, suppose that each $\Gamma_i$ is a Lipschitz Neighborhood Retract of an open neighborhood $U_i$. Denote by $\D$ the open unit disk in $\R^2$. Then, if

 \begin{equation*}
 \mathcal A = \{ F\in W^{1,2}(\D, A_Q(\R^n) )\cap C^0(\bar\D,A_Q(\R^n)) \ | \ F|_{S^1}(z)=\sum_i \sum_{\zeta = z^{k_i}}\llb f_i(\zeta) \rrb \ 
 \end{equation*}
\begin{equation*}
  \mbox{with $f_i:S^1\rightarrow \Gamma_i$ weakly monotonic} \}
  \end{equation*}

 \noindent then there is a map $f\in \mathcal A$ so that:

 \begin{center}
 $\Dir(f)=\ds\min_{G\in \mathcal A} \Dir(G)$.
 \end{center} \end{thm}

Recall that a map $f_i: S^1 \rightarrow \Gamma_i$ is said to be \emph{weakly monotonic} if it may be written as $\varphi_i ( e^{i\tau_i(\theta)})$, where $\varphi_i:S^1\rightarrow \Gamma_i$ is a homeomorphism, and $\tau_i : [0,2\pi] \rightarrow \R$ is a nondecreasing continuous, $2\pi$ periodic function.

Unfortunately, there is a subtlety in the existence of `wrapped' solutions--that is, those solutions with some $k_i > 1$. For example, if the optimal $f_i$ is constant on an arc of length greater than $\frac{ 2\pi }{k_i}$, there will be a branch point on the boundary circle $S^1$. Simple planar examples show that this can indeed occur--and so we must introduce a condition to guarantee such branching does not occur:

\begin{dfn}
\label{dfn:wrapped}
In the setting of Theorem \ref{thm:main}, the boundary data  $(\Gamma_1,...,\Gamma_K,k_1,...,k_K)$ is said to \emph{admit a wrapped solution} if there exists $f$ a minimizing element of $\mathcal A$ so that:

\begin{equation*}
f|_{S^1}(z)=\sum_i \sum_{\zeta = z^{k_i}}\llb f_i(\zeta) \rrb
\end{equation*}

\noindent for $f_i : S^1 \rightarrow \Gamma_i$ weakly monotonic so that no $f_i$ is constant on an arc of length greater than or equal to $\frac{2\pi}{k_i}$.

\end{dfn}

One should note that, while Theorem \ref{thm:main} guarantees a minimizing element of $\mathcal A$ under fairly general conditions, this element can fail to satisfy the condition of Definition \ref{dfn:wrapped} without contradicting its admissibility. 

Definition \ref{dfn:wrapped} is an analogous condition to the Douglas condition for planar domain Plateau solutions, as in \cite[8.6]{Hildmin} where similar degenerations must be ruled out. Interestingly, though, in contrast to the classical case, we are able to establish existence of a solution without this condition--but regularity fails.

As in the classical case, we begin with an admissible sequence $\{f^n\}$ whose energy tends towards the infimum and we extract a subsequence which converges ($W^{1,2}$ weakly) to some function $f: \D \rightarrow A_Q(\R^n)$. Then the issue is showing the admissibility of the function $f$--a fact which ultimately comes down to proving that the boundary data of the subsequence cannot collapse. In fact, a posteriori, a selection for the optimal boundary data will consist of homeomorphisms of the boundary; a stronger result than that $f|_{S^1}$ is a homeomorphism onto a subset of $A_Q(\R^n)$.

To prevent the boundary data from collapsing, one first precomposes with a disk automorphism to guarantee that, if $f^n|_{S^1}$ has boundary data $f^n_j$ for $f^n_j : S^1 \rightarrow \Gamma_j$ as in the definition of $\mathcal A$, then $f^n_1$ satisfies a three point condition. We will then show that (in the case that $f$ is branched) this normalization forces a normalization for each $f^n_j$--on a first attempt, one might try, supposing towards a contradiction, to analyze Dirichlet minimizing functions $G: \D \rightarrow A_Q(\R^n)$ for which $0\in \spt(G(s))$ for all $s\in S^1$. Unfortunately, such functions can exist with rather elaborate branching behavior, as exhibited in Section \ref{section:degen}, so a more complicated argument is required.

Finally, using recent results regarding the Plateau problem in abstract metric spaces (in particular, \cite{Disks}), in conjunction with Theorem \ref{thm:main}, we are able to prove the following.

\begin{thm}
\label{thm:conformal}
Suppose that $\Gamma_1,...,\Gamma_K$ and $k_1,...,k_K$ are given as in Theorem \ref{thm:main}, and further suppose that  the boundary data  $(\Gamma_1,...,\Gamma_K,k_1,...,k_K)$ admits a wrapped solution. Then there exists  $F : \D \rightarrow A_Q(\R^m)$, an area-minimizing Plateau solution for some collection of disjoint Jordan curves $\Gamma_1,...,\Gamma_K$, so that if $f_1,...,f_K$ is the associated boundary data the following holds:
\begin{enumerate}
\item around each regular point there exists a conformal selection for $F$
\item $\Dir(F) = \mbox{MV-Area}(F)$ and
\item Each $f_i$ is a homeomorphism.
\end{enumerate}

\noindent Here $\mbox{MV-Area}(F)$ is the two dimensional area of the set $\{\spt(F(x) \ | \ x\in \D\}$, counting multiplicity.

\end{thm}

This multi-valued Plateau problem may, as described in Section 6, lead to minimal surfaces necessarily having branch points, even in $\mathbb R^3$.   Here the location or number of the branch points is not fixed. Near a single prescribed branch point, a minimal surface has been studied by considering corresponding solution of a ``multi-valued minimal surface equation'' in  \cite{add3},  \cite{add1} and \cite{add2}. 

In Section \ref{section:topology},  we also present the sketch of a proof of Theorem \ref{thm:conformal} in the special case where $K=Q$ and $k_i=1$ using topological methods and solutions to the classical Douglas problem.

\section{Preliminaries}

A detailed introduction to the theory of multiple-valued functions may be found in \cite{revisit}. Here we present only the details essential to the current paper. We will mostly follow the notations found therein, with the exception that our definition of ``affinely approximatable'' is slightly weaker than the definition of ``differentiable'' given therein. Throughout the paper, $m$, $n$, and $Q$ will denote natural numbers, $\Omega$ is an open subset of $\R^m$ with sufficiently regular boundary, and, for $x$ an element of any metric space $X$, we will use $B_r(x)$ to denote the open ball of radius $r$ centered at $x$. We also denote by $\mathcal H^\alpha$ the $\alpha$-dimensional Hausdorff measure--in particular, $\mathcal H^0$ is counting measure.

We define the metric space $A_Q(\R^n)$ as the set of positive integer sums of Dirac measures on $\R^n$ with total mass $Q$. If, for $P \in \R^n$, $\llb P \rrb$ denotes the Dirac mass at $P$, then we may define $A_Q(\R^n)$ in symbols as follows:

\begin{equation*}
A_Q(\R^n) = \{\ds \sum_{i=1}^Q \llb P_i \rrb \ | \ P_i \in \R^n\}.
\end{equation*}

\noindent We note that we do not require the $P_i$'s to be distinct, and that we will occasionally omit $Q$ and $n$ when they are clear from the context. Following \cite{bigreg}, we metrize $A_Q(\R^n)$ by a metric $\mathcal G$ defined as:

\begin{equation*}
\mathcal G(\sum \llb P_i \rrb , \sum \llb Q_i \rrb ) = \min_{\sigma \in S_Q} \sqrt{\sum ||P_i - Q_{\sigma(i)}||^2}.
\end{equation*}

\noindent Where $S_Q$ denotes the permutation group on $\{1,...,Q\}$. It is easy to see that this gives a well-defined metric under which $A_Q(\R^n)$ is a complete metric space.

For an open $\Omega\subset \R^m$, we say a function $f: \Omega \rightarrow A_Q(\R^n)$ is a multiple-valued function. If $g : \Omega \rightarrow A_{Q'}(\R^n)$ is another multiple-valued function, we denote by $\llb f \rrb + \llb g \rrb$ the function into $A_{Q+Q'}(\R^n)$ whose value at the point $x\in \Omega$ is the sum of the measures $f(x)$ and $g(x)$.

 If $f_i:\Omega \rightarrow \R^n$ are such that $f=\sum \llb f_i\rrb$, we say $\{f_i\}$ is a selection for $f$. A major problem in the theory of multiple-valued maps is attempting to derive selections with some degree of regularity from multiple-valued functions. However, we always have the following result:

\begin{thm}
If $f: \Omega \rightarrow A_Q(\R^n)$ is measurable, then there exists a selection for $f$ consisting of measurable functions.
\end{thm}

Almgren developed an extrinsic theory through the help of a bi-Lipschitz embedding of $A_Q(\R^n)$ onto a Lipschitz retract of higher dimensional Euclidean space. A corollary attributed to B. White strengthens this embedding to locally preserve distances, and we present this version below.

\begin{thm}
\label{thm:embedding}

There exists $N=N(Q,n)$, an injective map $\zeta : A_Q(\R^n) \rightarrow \R^N$ and a Lipschitz retraction $\rho : \R^N \rightarrow \zeta(A_Q(\R^n))$ so that:

\begin{enumerate}
\item $\lip(\zeta) \leq 1$.
\item If $\mathcal Q = \zeta(A_Q(\R^n))$, then $\lip(\zeta^{-1}|_{Q}) \leq C(Q,n)$.
\item For any $T\in A_Q(\R^n)$, there exists $\delta$ so that $\zeta|_{B_\delta(T)}$ preserves distances.
\end{enumerate}

\end{thm}

This embedding allows us to quickly define multiple-valued Sobolev functions. For $\Omega \subset \R^m$, $k\in \mbb N$ and $1 \leq p \leq \infty$, we define $W^{k,p}(\Omega, A_Q(\R^n))$ to be those functions $f: \Omega \rightarrow A_Q(\R^n)$ so that $\zeta \circ f \in W^{k,p}(\Omega, \R^N)$. While this definition may seem rather contrived, it has many computational benefits to the more natural equivalent definitions discussed in \cite{revisit}. In particular, we note that classical results regarding Lipschitz approximation of Sobolev functions, interpolation lemmas, traces and one-dimensional restrictions of Sobolev functions are seen to immediately carry over to the multiple-valued case.

We further define the \emph{Dirichlet energy} of a $W^{1,2}(\Omega, A_Q(\R^n))$ function as:

\begin{equation*}
\Dir(f,\Omega) = \int_\Omega |D(\zeta \circ f)|^2.
\end{equation*}

\noindent Again, this is simply a pragmatic definition--more natural equivalent definitions are available in the literature.

\begin{rmk}
\label{rmk:invar}
From this definition it immediately follows that if $\Omega$ is two dimensional, the Dirichlet energy is invariant under precomposition by conformal maps.
\end{rmk}
We say that a multiple-valued function $f\in W^{1,2}(\Omega, A_Q(\R^n))$ is \emph{Dirichlet minimizing} if

\begin{equation*}
\dir(f,\Omega) \leq \dir(g,\Omega)
\end{equation*}

\noindent for any $g\in W^{1,2}(\Omega, A_Q(\R^n))$ such that $f|_{\partial \Omega} = g|_{\partial \Omega}$. Here, the last condition means that the $W^{1,2}$ traces on $\partial\Omega$  of $\zeta\circ f$ and $\zeta\circ g$ coincide. In Almgren's original work, he proved that if $g\in W^{1,2}(\Omega, A_Q(\R^n))$, then there exists a Dirichlet minimizing $f\in W^{1,2}(\Omega, A_Q(\R^n))$ with $f|_{\partial \Omega} = g|_{\partial \Omega}$. He further showed that there is a constant $\alpha=\alpha(m,Q) >0$ so that such a minimizing function is equal almost everywhere to a function which is $\alpha$ H\"older continuous on any $\Omega'\subset\subset\Omega$. In the following, we assume that this identification has been made, and in particular that Dirichlet minimizers are continuous functions ``strictly defined'' at every point.

In case $m=1$, Sobolev functions automatically admit an absolutely continuous selection. We may define the space of absolutely continuous multiple-valued functions on an interval $I$, denoted by $AC(I,A_Q(\R^n)$, to be those functions $f:I\rightarrow A_Q(\R^n))$ so that $\zeta \circ f$ is absolutely continuous. We may now state a selection theorem for one dimensional multiple-valued Sobolev functions--see Proposition 1.2 of \cite{revisit}:

\begin{thm}
If $f\in W^{1,p}(I,A_Q(\R^n))$, then,
\begin{enumerate}
\item $f\in AC(I,A_Q(\R^n))$.
\item There exists a selection $\{f_i\} \subset W^{1,p}(I,\R^n)$ and so that $|Df_i| \leq |D(\zeta\circ f)|$.
\end{enumerate}
\end{thm}

Two more definitions will be useful in the following.

\begin{dfn} We say a multiple-valued function $f:\Omega \subset \R^m \rightarrow A_Q(\R^n)$ is \emph{affinely approximatable} at $x_0\in \Omega$ if there exist affine maps $T_i: \R^m \rightarrow \R^n$ for $i=1,...,Q$ so that:

\begin{enumerate}

\item $\mathcal G (f(x_0), \sum \llb T_i(x)\rrb ) =o(|x-x_0|)$

\item If $f_i(x_0) = f_j(x_0)$, then $L_i=L_j$.

\end{enumerate}
\end{dfn}

\begin{dfn}
A \emph{branch point} for a multiple-valued function $f: \Omega \rightarrow A_Q(\R^n)$ is a point of discontinuity for the function $\sigma(x)=\Ha^0(\spt(f(x)))$.
\end{dfn}

\begin{rmk} For a Dirichlet minimizer the set of branch points is at most of Hausdorff dimension $m-2$, and, in case $m=2$, it is locally finite. If $x\in \Omega$ is not a branch point for a Dirichlet minimizing $f$, then in some neighborhood around $x$, there exists a harmonic selection for $f$, and such a point is called a \emph{regular point} for $f$. Note this definition allows the possibility that at a regular point one of the selection functions for $f$ may still have a critical point where the rank of its differential is less than $\min\{ m,n\}$. 
\end{rmk}

\section{A Multiple-Valued Plateau Problem}

In this section, we will prove Theorem \ref{thm:main}. Notice that for $K=1$, $k_1=1$, this theorem reduces to the classical case proven by Douglas and Rado, see \cite{Douglas} and \cite{Rado}.

The naive idea is the following: take a minimizing sequence and apply a version of the Courant-Lebesgue Lemma to an appropriate subsequence to get a set of maps which are weakly monotonic on the boundary and then appeal to lower semi-continuity of Dirichlet energy under weak convergence. Unfortunately, the Courant-Lebesgue Lemma requires the so called ``three point condition.'' In the classical case, this condition is obtained by ``normalizing" the functions by precomposing with an automorphism of the disk. In our context, the problem with this approach is that we cannot a priori simultaneously normalize all of the boundary functions--the simplest explanation for our method around this is when $K=2$. We will show that, if there are certain kinds of branch points within the disk for each element of a minimizing set, then normalization of one of the boundary functions automatically normalizes the other. In the absence of these branch points, we may normalize separately.

To do this, we require two lemmas. The first lemma relies upon the following (slight) generalization of two Theorems of W. Zhu, found in \cite{WZhu2dim}.


\begin{thm}
\label{thm:Zhu}
Let $\epsilon >0$ and suppose an energy minimizing $f\in W^{1,2} ((1+\epsilon)\D, A_Q(\R^n))$ is strictly defined and has precisely one branch point in $\D$ at the origin. Then there exists $j \leq Q$ and $k_1, k_2, ..., k_j \in \mathbb N$ so that $\kappa := \sum k_i \leq Q$ and harmonic functions $f_1,...f_{k_j},f_{Q-\kappa},f_{Q-\kappa+1},...,f_Q:\D \rightarrow \R^n$ so that, for all $s\in \D$.

\begin{equation*}
f(s) = \sum_{i=1}^Q \left(    \sum_{\zeta=s^{k_i}} \llb f_i(\zeta)\rrb\right)
\end{equation*}

\noindent where, to suppress notation, we set $k_{\ell} = 1$ for $\ell \geq j$.

\end{thm}

The next lemma is elementary, but, to the authors's knowledge, isn't explicitly stated in the available literature.

\begin{lemma}
\label{lemma:sepcont}
Suppose $\Omega\subset \R^m$ is path connected and locally path connected and that $f: \Omega \rightarrow A_Q(\R^n)$ is continuous. Additionally, suppose that $f$ satisfies the following for some $\alpha >0$:

\begin{enumerate}
\item $\mathcal H^0(\mbox{\emph{spt}}(f(x)))=Q$ for all $x\in A$
\item $\{ (x,y) \in \Omega \cross \R^n \ | \ y\in \spt\{f(x)\} \}$ has $K$ connected components $E_1,...,E_K$ so that $d(E_i,E_j) > \alpha$ for $i\neq j$.
\end{enumerate}

\noindent  Then $f$ admits a continuous decomposition, that is, there are integers $1', 2',...K'$ so that $\sum i'=Q$, continuous functions $f_j : \Omega \rightarrow A_{j'}(\R^n)$ so that $f = \sum \llb f_i \rrb$. Moreover, if $\Omega$ is a region in $\R^n$ and $f\in W^{1,p}(\Omega,A_Q(\R^n))$ for some $1<p<\infty$, then each $f_i \in W^{1,p}(\Omega, A_{i'}(\R^n))$.

\end{lemma}

\begin{proof}

 Let 
 
 \begin{equation*}
 Gf(z) = \ds\sum_{y\in \spt(f(z))} \llb (z, y) \rrb
 \end{equation*}
 
 \noindent be the graph map associated to $f$. Note that if $f$ is continuous (or Sobolev), so too is $Gf$. Let $i'$ be the number of elements of $E_i \cap\spt(Gf(z))$. We will show that $i'$ is indendent of $z \in \Omega$.
 
Let $x, y\in \Omega$ be given and let $c:[0,1] \rightarrow \Omega$ be a curve with $c(0)=x$ and $c(1) = y$. Since assumptions (1) and (2) guarantee that $f$ has no branch branch points on $c([0,1])$, well known selection theory (see, for example the beginning chapters of \cite{revisit}) gives a unique selection for $Gf\circ c$ as 

\begin{equation*}
Gf\circ c = \sum \llb g_i \rrb.
\end{equation*}

We see that, if $g_i(0) \in E_\ell$, then $d(g_i(t),E_k) > \alpha$ for all $k\neq \ell$ by assumption two. Therefore, $g_i(1) \in E_\ell$, proving $i'$ is well defined.

Thus for $y\in \Omega$, we may put:

\begin{equation*}
f_i(y) = \sum_{(y,p)\in E_i} \llb p \rrb.
\end{equation*}

To see that each $f_i$ is continuous, given $x\in A$ and  $\epsilon > 0$, we choose $\delta$ so that $|x - y| < \delta$ implies $\mathcal G (Gf(x),Gf(y)) < \min(\epsilon, \tfrac{\alpha}{3} )$.  Let $Gf(x) = \sum \llb (x,p_i)\rrb$ and $Gf(y) \sum \llb (y,q_i) \rrb$ be labeled so that:

\begin{equation*}
\mathcal G(Gf(x), Gf(y) )^2 =\sum_i ||x-y||^2 + ||p_i - q_i ||^2.
\end{equation*}

\noindent Then, in particular, for each $i$, $||x-y||^2 + ||p_i - q_i ||^2  < \alpha^2$. So, $(x,p_i) \in E_j$ implies that $(y,q_i) \in E_j$, which implies that $\mathcal G(f_j(x) , f_j(y) ) < \epsilon$, so each $f_j$ is continuous.

Finally we show that each $f_j$ must be Sobolev as follows: for a.e. point $x$ in the interior of $\Omega$, if $\ell(x)$ is a line parallel to the unit axis of $\R^n$ passing through $x\in\R^n$, then $Gf|_{\ell(x)}$ is Sobolev on some neighborhood $U\subset \ell(x)$ of $x$. This implies, by Proposition 1.2 of \cite{revisit} that there exists a Sobolev selection on $U$. However, since each $f_i$ is continuous  we must have that the Sobolev selection corresponds to a Sobolev selection for each $f_i|_U$.  Hence by Theorem 2  of \cite[\S 4.9.2]{EandG} each $f_j \in W^{1,2}_{loc}(\Omega, A_{j'}(\R^n))$, and $f\in W^{1,2}(\Omega, A_Q(\R^n))$ allows us to conclude that $f_j \in W^{1,2}(\Omega,A_{j'}(\R^n))$ for all $j$.

\end{proof}

Our ultimate goal is to apply the following generalization of the Courant-Lebesgue Lemma to the boundary data of a minimizing sequence of multiple-valued functions. For the result in the classical setting, see \cite[\S 4.3 Theorem 3]{Hildmin}.

\begin{thm}
\label{thm:courant}
Suppose $\Gamma_1,...\Gamma_K$ are  closed, oriented Jordan curves as in the statement of Theorem \ref{thm:main} and that $\{f^i_k\}_{k\in \mbb N}$ are sequences of monotone continuous mappings $f^i_k: S^1\rightarrow \Gamma^i$. Define $F_k : S^1 \rightarrow A_Q(\R^n)$ by:

\begin{equation*}
F_k(s)=\sum_{i} \sum_{\zeta = s^{k_i}}\llb f^i_k(\zeta) \rrb
\end{equation*} 

\noindent  and suppose that the Dirichlet minimizing extensions of the maps $F_k$, denoted by $\td F_k: \D \mapsto A_Q(\R^n)$ are such that

\begin{center}
\Dir$(\td F_k;\D) \leq M$
\end{center}

\noindent for some $M\in\R$ independent of $k$, then, for each $i$, the family $\{f^i_k\}_{k\in \mbb N}$ is equicontinuous if they satisfy a uniform three point condition:

\begin{center}
$f^i_k(\alpha^i_j) = \beta^i_{k,j} $ for $j=1,2,3,$, $i=1,...,Q$
\end{center}

\noindent for some distinct points $\alpha^i_j\in S^1$ and $\beta^i_{k,j} \in \Gamma^i$ so that $\beta^i_{k,j} \rightarrow \beta^i_j$ holds, for $\beta^i_j$ three distinct points of $\Gamma_i$.

\end{thm}

\noindent Once we guarantee that we can find an $A$ (independent of $k$) for which the conditions of Lemma \ref{lemma:sepcont} hold for each $\td F_k$, the proof of Theorem \ref{thm:courant} follows from identical methods of \cite[\S 4.3 Theorem 3]{Hildmin} after applying Lemma \ref{lemma:sepcont} and analyzing the embeddings $\zeta \circ f_i$ for the decomposition guaranteed by the lemma. Finding such a neighborhood is the content of the next proposition.

\begin{lemma}
\label{lemma:nbd}
Suppose that $\{F_k\}_{k\in \mathbb N} \subset W^{1,2}(\D, A_Q(\R^n))\cap C(\overline{\D}, A_Q(\R^n))$ is a sequence of Dirichlet minimizing functions with a selection for the trace of $F_k$ given, for $s\in S^1$ by

\begin{equation*} 
F_k(s)= \sum_{i} \sum_{\zeta = s^{k_i}}\llb f^i_k(\zeta) \rrb
\end{equation*}

\noindent  for $f_k^i : S^1 \rightarrow \Gamma_i$ continuous, weakly monotonic maps onto curves as in Theorem \ref{thm:main}. Further suppose that each $F_k$ is affinely approximatable a.e. on the interior of $\D$ and that

\begin{equation*}
\sup_k \Dir(F_k;\D)=M < \infty.
\end{equation*}

\noindent  Then there exists a neighborhood $A\subset \D$ of $S^1$ so that the conditions of Lemma \ref{lemma:sepcont} are satisfied for each $F_k|_{A}$.

Further, if each $\Gamma_i$ is a LNR of some open neighborhood $U_i$ (where we assume that $U_i \cap U_j = \phi$ for $i\neq j$), we may guarantee that: $F_k|_A \in \sum_1^K k_i \llb U_i \rrb$, where $\sum_1^K k_i \llb U_i \rrb$ is the set of all $y\in A_Q(\R^n)$ so that $\mathcal H^0 (\spt(y) \cap U_i )=k_i$ for all $i$.

\end{lemma}

\begin{rmk}
Note that the differentiability condition is immediately satisfied by any sequence of functions which are Dirichlet minimizing for their boundary data, as in \cite{bigreg} or \cite{Delint}.
\end{rmk}

\begin{proof}

Let $\epsilon < \ds \min_{i\neq j} \text{dist}(\Gamma_i,\Gamma_j)$. We will show that for $\delta <1$ close enough to one and any $x\in S^1$,  $\delta \leq \alpha \leq 1$, $\mathcal G(F_k(\alpha x),F_k( \gamma )) < \frac{\epsilon}{3}$, where $\gamma$ is some arc of $S^1$, independent of $k$, and $x\in\gamma$.  By our assumptions on the boundary data of each $F_k$, this will guarantee the assumptions of Lemma \ref{lemma:sepcont}. We will assume that $M=1$, which is possible simply by dividing the functions appropriately.

 Let $1 \geq  \alpha \geq \delta \geq \frac{1}{2}$ where $\delta$ will be chosen shortly, and $0 < \Delta \theta \leq \frac{\pi}{2}$ be given. Let $a \in \alpha\cdot S^1$ be given, and let $\hat a=\frac a {||a||}$. Viewing $\R^2$ as $\C$, we lose no generality in assuming that $\hat a =1$. Consider the region $R_\delta$ defined by:

\begin{equation*}
R_\delta= \{ (x,y) | \ y\in [-\delta\sin(\frac{\Delta \theta}{2}),\delta\sin(\frac{\Delta\theta}{2})], \mbox{and} \ \delta^2 \leq x^2+y^2 \leq 1 \}.
\end{equation*}

\noindent For each $y_0\in [-\delta\sin(\frac{\Delta \theta}{2}),\delta\sin(\frac{\Delta \theta}{2})]$, let $L(y_0)$ be the horizontal line segment $(\R\cross \{y_0\}) \cap R_\delta$. Denote by $x_{\delta}(y_0)$ and $x_1(y_0)$ the $x$ coordinates of where this line intersects $\delta S^1$ and $S^1$, respectively.

Since each $F_k$ is Sobolev, for a.e.  line $\ell$ in $\C$, we have that $F_k|_{(\ell \cap \D)}$ is absolutely continuous for all $k$ and that $F_k$ is affinely approximatable at $\Ha^1$ a.e. point on $\ell$. This implies, by combining Proposition 1.2 of \cite{revisit} and Theorem 6.4 of \cite{selection}\footnote{While the statement of the theorem in \cite{revisit} requires that $F_k$ be Lipschitz, the proof is identical provided that $F_k$ admits an absolutely continuous selection, which is guaranteed by Proposition 1.2 of \cite{revisit}. } that for a.e. $y_0\in [-\delta \sin(\frac{\Delta \theta}{2}),\delta\sin(\frac{\Delta \theta}{2})]$
\begin{equation*}
\mathcal G(F_k(x_{\alpha}(y_0),y_0), F_k(x_1(y_0),y_0))\leq \int_{L(y_0)\cap \{ x \ | \ \delta \leq ||x|| \leq 1\} } ||\nabla F_k || \, d\Ha^1 \leq \int_{L(y_0)} ||\nabla F_k || \, d\Ha^1.
\end{equation*}

\noindent Integrating this expression with respect to the variable $y$ gives:

\begin{equation*}
\int_{-\delta\sin(\frac{\Delta \theta}{2})}^{\delta\sin(\frac{\Delta \theta}{2})}\mathcal G(F_k(x_{\\delta}(y),y), F_k(x_1(y),y)) \, dy\leq \int_{-\delta\sin(\frac{\Delta \theta}{2})}^{\delta\sin(\frac{\Delta \theta}{2})}\left(\int_{L(y)} ||\nabla F_k || \, d\Ha^1\right) \, dy.
\end{equation*}

\noindent The Fubini Theorem and the H\"older Inequality implies:
\begin{equation}
\label{eq:rhsbound}
\int_{-\delta\sin(\frac{\Delta \theta}{2})}^{\delta\sin(\frac{\Delta \theta}{2})}\left(\int_{L(y)\cap R_\delta} ||\nabla F_k || d\Ha^1\right) \, dy\leq M^{\frac{1}{2}} \left(\Ha^2(R_\delta)\right)^{\frac1 2}.
\end{equation}

\noindent On the other hand, if $\gamma_1$ and $\gamma_2$ denote the portions of $\alpha\cdot S^1$ and $S^1$ contained in $R_\delta$, we see that:

\begin{equation*}
{2\delta\sin(\frac{\Delta \theta}{2})} \mathcal G(F_k(\gamma_1),F_k(\gamma_2)) \leq \int_{-\delta\sin(\frac{\Delta \theta}{2})}^{\delta\sin(\frac{\Delta \theta}{2})}\mathcal G(F_k(x_{\alpha}(y),y), F_k(x_1(y),y)) \, dy
\end{equation*}

\noindent which, combined with Equations (\ref{eq:rhsbound}) and (\ref{eq:choice}) gives the bound:

\begin{equation}
\label{eq:neccest}
\mathcal G(F_k(\gamma_1),F_k(\gamma_2)) \leq \frac{\Ha^2(R_\delta)^{\frac1 2} }{2\sin(\frac{\Delta \theta}{2})\delta} 
\end{equation}

 (Recall that we've assumed $M=1$.) By rotation, this bound may be established for any arc of length equal to $\Delta \theta$. 
  
Next, one easily checks that $\Ha^2(R_\delta)\rightarrow 0$ as $\delta \rightarrow 1$ so we may choose (for fixed $\Delta \theta$) $\delta$ close enough to one so that:

\begin{equation}
\label{eq:choice}
\frac{\Ha^2(R_\delta)^{\frac1 2} }{2\sin(\frac{\Delta \theta}{2})\delta} \leq \frac{\epsilon}{6}.
\end{equation}
 
We now observe that, if the arc length were chosen to be small enough so that the variation of all of the functions $F_k$ on $\delta \gamma$ was smaller than $\frac \epsilon 6$, we could guarantee that, for each $e^{i\psi} \in \gamma$:

\begin{equation*}
\mathcal G(F_k(\delta e^{i\psi}, F_k(\gamma) ) < \frac \epsilon 3.
\end{equation*}

Obtaining such a bound for smaller arc length is slightly more complex than it seems at first, since $\delta$ must be chosen \emph{after} the arc-length is already fixed, but it can be done via the following method of choosing $\delta$ and the arc length $\ell$.

Recall that, in Theorem 9 of \cite{revisit}, the following inequality is shown for Dirichlet minimizing functions on the disk:

\begin{equation}
\label{eq:crucial}
\int_{B_r(x)} ||DF_k||^2 \, d\h^2 \leq r^{\frac 2 Q} \int_{\D} ||Df||^2 \, d\h^2 
\end{equation}

\noindent Whenever $x\in \D$ and $r < 1-||x||$. 

Following the arguments of Section 3.2 of \cite{LinPDE}, we see that this implies, some constant $C$ which depends only on $n$ and $Q$ so that:

\begin{equation*}
x,y \in \delta \D \mbox{ with } ||x-y||< \frac{1-\delta}{2} \Rightarrow \mathcal G (F_k(x),F_k(y))\leq C ||x-y||^{\frac 1 Q}.
\end{equation*}

Let $\epsilon$ be as above and choose $n\in \mbb N$ so that $\frac{2\cdot7^2}{ \epsilon^2 } \leq n $.  For $\delta$ sufficiently close to one,  $\ell = n\left(\frac{1-\delta}2\right)$ defines an arc within the ranges given in the preceding arguments. 
 
Further, we may also choose $\delta$ close enough to one to guarantee the additional inequality:

\begin{equation}
\label{eq:choice2}
n C \left( \frac{1-\delta}{2} \right) ^{\frac 1 Q} \leq \frac \epsilon 6.
\end{equation}
 
 Then, using the asymptotics:

\begin{equation*}
 \h^{ 2}(R_\delta)  \approx \ell(1-\delta)
 \end{equation*}
 \begin{equation*}
2\sin(\frac{\Delta \theta}{2} ) \approx \ell
 \end{equation*}
 
 \noindent and our choices of $n$, $\delta$ and $\ell$, one checks that the above integral estimates imply (possibly upon choosing $\delta$ even closer to one):

\begin{equation*}
\mathcal G(F_k(\gamma),F_k(\delta \gamma)) \leq \frac{\epsilon}{6}.
\end{equation*}

\noindent for any arc $\gamma \subset S^1$ with length $\ell$. 

Let $\gamma$ be such an arc. For any two points $x, y\in \delta \gamma$, by choice of $\ell$, we may find a chain :

\begin{center}$\{x=x_1,x_2,...,x_n=y\}\subset \delta \gamma$ for which $||x_i - x_{i+1}|| \leq \frac{1-\delta}{2}$. 
\end{center}

\noindent Then:

\begin{equation*}
\mathcal G(F_k(x),F_k(y) ) < \sum \mathcal G(F_k(x_i),F_k(x_{i+1})) \leq n C \left(\frac{1-\delta}{2}\right)^{\frac 1 Q} \leq \frac{\epsilon}{6}
\end{equation*}

\noindent where the last inequality follows from Inequality (\ref{eq:choice2}).

To summarize, we have shown the following statement: 

\begin{center}
{\it If $\gamma$ is an arc of $S^1$ of length $\ell$ and $x \in \gamma$, $k\in \mbb N$, there exists $y_k\in \gamma$ so that $\mathcal G(F_k(\delta x) , F_k( y_k) ) \leq \frac{\epsilon}{3}$.} 
\end{center}

One checks that this statement proves the initial goal, since choosing $\delta \leq \alpha$ only improves the bounds used above. By further constraining $\epsilon$ to also satisfy $\epsilon < \ds \min_i \text{dist}(\Gamma_i, \partial U_i)$, we may obtain the final statement in the Lemma, completing the proof.

\end{proof}

\begin{remark}
\label{rmk:addit}
Note that Equation (\ref{eq:neccest}) a posteriori gives the additional property, where $\{f_i\}$ is the selection from Lemma \ref{lemma:sepcont}:  There is an $\alpha >0$ such that for all $0 < \alpha \leq \beta \leq 1$, $\beta S^1 \subset A$ and , and all $i$, $\mathcal G (f_i( \beta e^{i\theta}) ,f_i(\gamma)) < \epsilon$ where $\gamma$ is an arc of $S^1$ with length bounded by $C(1-\beta)$ with $e^{i\theta} \in \gamma$.
\end{remark}

Finally, we recall the following interpolation lemma, Lemma 2.15 of \cite{revisit}.

\begin{lemma}[Interpolation Lemma]
\label{lemma:interpolation}
There is a constant $C=C(m,n,Q)$ with the following property. Let $r, \ \epsilon > 0, \ g\in W^{1,2}(\partial B_r, A_Q(\R^n))$ and $f\in W^{1,2}(\bd B_{r(1-\epsilon}), A_Q(\R^n))$. Then, there exists $h\in W^{1,2}(B_r\setminus B_{r(1-\epsilon)}, A_Q(\R^n))$ so that $h|_{\bd B_{r(1-\epsilon)}} = f$, $h|_{\bd B_{r}} = g$ and so that $\dir(h,B_r\setminus B_{r(1-\epsilon)}) $ is less than or equal to:

\begin{equation}
C\epsilon r [\dir(g,\partial B_r) + \dir(f,\partial B_{r(1-\epsilon)})] + \frac{C}{\epsilon r} \int_{\partial B_r} \mathcal G(g(x), f((1-\epsilon)x))^2 \, dx
\end{equation}

\end{lemma}

\begin{proof}[Proof of Theorem \ref{thm:main}] We present a proof in the case where $K=2$, and $k_1$ and $k_2$ are given so that $k_1 + k_2 = Q$. The more general case follows by identical arguments, there are simply more cases to analyze. For more details, see Remark \ref{rmk:higherq}.

Let $\{F_k\}_{k=1}^\infty\subset W^{1,2}(\D, A_2(\R^n))$ be a minimizing sequence so that:

\begin{enumerate}
\item $F_k |_{S^1} (s)= \sum_{\zeta=s^{k_1}} \llb f_k(\zeta) \rrb +  \sum_{\zeta=s^{k_1}}\llb g_k (\zeta)\rrb $ for $f_k : S^1 \rightarrow \Gamma_1$ and $g_k : S^1 \rightarrow \Gamma_2$ monotone. 
\item $\dir(F_k, \D ) \rightarrow \ds \inf_{G\in \mathcal A} \dir(G,\D)$ 
\end{enumerate}

We may additionally assume that each $F_k$ is Dirichlet minimizing for its boundary data--since replacing each $F_k$ with a Dirichlet minimizer will only decrease the energy.

Standard compactness arguments allow us to extract a subsequence (which we won't relabel) which weakly converges in $W^{1,2}$ to a Dirichlet minimizing function $F$. Furthermore, by Arzela-Ascoli, uniform Holder continuity on any compact $\Omega \subset \subset \D$ and a diagonal argument applied to a compact exhaustion of $\D$,  we may guarantee that the $f_k$ converge uniformly on compact subsets of the disk.

Let $A$ be the neighborhood of $S^1$ provided by Lemma \ref{lemma:nbd}. For simplicity, we assume that $A$ is a disk of inner radius $\delta$. For reasons that will be elucidated shortly, we will mostly be interested in the branch points of $F$ within $\D \setminus A$, so  denote by $\Sigma_k$ the branch set of $F_k$ within the compact set $\D \setminus A$ and $\Sigma_0$ the branch set of $F$ within $\D \setminus A$. Since interior branch points for 2-dimensional multiple-valued functions are isolated, we know that $F$ may have only finitely many branch points in $\D \setminus A$. Further, since the $F_k$'s converge to $F$ uniformely on $\D \setminus A$, we may, by throwing out finitely many elements of the sequence $\{F_k\}$, also assume that $\mathcal H^0(\Sigma_k) \geq \mathcal H^0 (\Sigma_0)=\ell$. Put $\Sigma_0 =\{x_1,...,x_\ell\}$.

For the remainder of the proof it will be crucial to have more control on the location and quantity of points in $\Sigma_k$--the following claim shows that we can do this without changing our sequence's properties substantially.

\begin{claim}
\label{claim:modification}
We may modify the sequence $\{F_k\}$ to obtain a new sequence, $\{\td F_n \}$ still converging uniformly on compact sets to $F$ so that if $\td \Sigma_n$ is the branch set of $\td F_n$ in $\D \setminus A$, then $\mathcal H^0 (\td \Sigma_n) = \ell$ and so that $\dir(\td F_n) = \dir( F_{k(n)}) +o(1)$ as $n\rightarrow \infty$ for some increasing sequence $k(n)$.  
\end{claim}

\begin{proof}[Proof of Claim]

Using the uniform convergence of the sequence $\{F_k\}$ on $\D \setminus A$, we may find $r_1,...,r_\ell$ so that the following holds for all large $k$:

\begin{enumerate}
\item $B_{r_i}(x_i) \cap B_{r_j}(x_j) = \phi$ for $i\neq j$. 
\item $\Sigma_k \subset \cup B_{r_i}(x_i)$
\item The image of the graph map of $F_k |_{\partial B_{r_i}(x_i)}$ has the same number of connected components for all $k$.
\end{enumerate}

\noindent For each $k\in \{1,...,\ell\}$, we consider $B_{r_i}(x_i)$ independently. For ease of notation, assume that $x_i=0$, and, as before, let $B_s$ be the open ball of radius $s$ about the origin.

Let the image of the graph map of $F_k |_{\partial B_{r_i}}$ have $\ell$ connected components.

Then there are positive integers $\ell_{i,j}$ for $j=1,...,\ell$ so that $\sum_j\ell_{i,j}=Q$   and there are $\ell$ continuous functions $f_{k,i,j}$ defined on $\partial B_{r_i}$ mapping to $A_{\ell_{i,j}}(\R^n)$  so that:
\begin{equation}
F_k|_{\partial B_{r_i}} =\sum_j  \llb f_{k,i,j} \rrb
\end{equation}

\noindent and so that $\mathcal G (f_{k,i,j} ,f_{k,i,m})$ is bounded away from zero independent of $k$ for $j\neq m$. Since $F$ is Dirichlet minimizing, has only one branch point within $B_{r_i}(x_i)$, and has $\ell$ connected components on $\bd B_{r_i}$,  we conclude via Theorem \ref{thm:Zhu} that there are $\ell$ harmonic functions $h_j : B_{1} \rightarrow \R^n$ so that:

\begin{equation*}
 F|_{B_{r_i}}(z) = \sum_j \sum_{\zeta^{\ell_{i,j}}=r_i z} \llb h_j(\zeta) \rrb.
 \end{equation*}
 
\noindent Uniform convergence implies that (up to relabeling),  
  
  \begin{equation*}
  f_{k,i,j}(z) \rightarrow \text{Trace}\left(\sum_{\zeta^{\ell_{i,j}}=r_i z} \llb h_j(\zeta) \rrb\right )
  \end{equation*}

  \noindent uniformly as $k\rightarrow \infty$. For arbitrary $0 < \epsilon < r_i$, we may apply Lemma \ref{lemma:interpolation} to the following sequence of functions:

\begin{enumerate}
\item $F_k|_{\partial B_{r_i}} =\sum_j \llb f_{k,i,j} \rrb$
\item $F_\epsilon(s)=F(\frac{s}{1-\epsilon})$ defined on $\bd B_{r_i(1-\epsilon)}$.
\end{enumerate}

\noindent We obtain a sequence $h_{k,\epsilon} \in W^{1,2} (B_{r_i}\setminus B_{r_i(1-\epsilon)}, A_Q(\R^n))$ with bounds on $\dir(h_{k,\epsilon},B_{r_i}\setminus B_{r_i(1-\epsilon)}) $; in particular:

\begin{equation}
\label{eq:bounds}
\dir(h_{k,\epsilon},B_{r_i}\setminus B_{r_i(1-\epsilon)})  \leq C\epsilon r [\dir(F_k,\partial B_{r_i}) + \dir(F_{\epsilon},\partial B_{r_i(1-\epsilon)})] + \frac{C}{\epsilon r} \int_{\partial B_{r_i}} \mathcal G(F_k(x), F_\epsilon((1-\epsilon)x))^2 \, dx.
\end{equation}

Notice that in the above equation, as $k \rightarrow \infty$, the right hand side tends to:

\begin{equation}
2C\epsilon r \dir(F_\epsilon,\partial B_{r_i}) \leq M \epsilon
\end{equation}

\noindent where $M$ is independent of $\epsilon$. 

Let $\epsilon_n$ be any sequence tending to zero, and further suppose that $0 < \epsilon_n < \min(r_i,\epsilon)$. For each $n$, choose an increasing sequence $k(n)$ so that:

\begin{equation*}
C\epsilon_n r [\dir(f_{k(n)},\partial B_{r_i}) + \dir(h_{\epsilon_n},\partial B_{r_i(1-\epsilon_n)})] + \frac{C}{\epsilon_n r} \int_{\partial B_{r_i}} \mathcal G(f_{k(n)}(x), h_{\epsilon_n}((1-\epsilon)x))^2 \, dx < 2M\epsilon.
\end{equation*}

\noindent Now, define $\td F_n$ by:

 \begin{displaymath}
   \td F_n(x) = \left\{
     \begin{array}{llr}
       F_{k(n)} (x) & : & x \in \D \setminus B_{r_i}\\
       h_{k(n),\epsilon_n} (x) & : & x\in B_{r_i} \setminus B_{r_i(1-\epsilon)}\\
       F (\frac{x}{1-\epsilon_n}) & : & x\in B_{r_i(1-\epsilon)}\\
     \end{array}
   \right.
\end{displaymath}

\noindent One may check that $\dir(\td F_n, B_{r_i}) = \dir( F_k(n),B_{r_i}) +o(1)$ holds, and further that $\td F_n$ has only one branch point within $B_{r_i}$ for sufficiently large $n$--namely $0$. 

By applying the above construction to each $B_{r_i}(x_i)$, we obtain a new sequence $\{\td F_k\}$ satisfying the conclusions of the Claim.

\end{proof}

\begin{remark}
\label{rmk:local}
Note that the new functions $\td F_k$ are no longer Dirichlet minimizing for their given boundary data, however, they are \emph{locally} Dirichlet minimizing in some neighborhood of their branch points which is independent of $k$.  
\end{remark}

There is a sort of dichotomy between the two types of branch points that may occur in $\Sigma_k$, which we now analyze.

\begin{dfn}
Call $\{x_i^k\}_{k\in \mbb N}$ \emph{connecting} if there is a smoothly bounded neighborhood $U$ so that:

\begin{enumerate}
\item for all large $k$, $U\cap \Sigma_k = \{x_1^k\}$ and,  
\item $U$ is smoothly contractible to $x_1^k$ for all $k$ as in (1),
\item $\partial U$ is a smooth curve whose intersection with $A$ has a smooth arc $\Gamma$ of positive $\Ha^1$ measure
\item  If $GF_k$ denotes the graph map (as in the proof of Lemma \ref{lemma:sepcont}), then $\spt(GF_k(\partial U)$ has a connected component $\Lambda$ so that $(\Lambda \cap \spt(GF_k(\Gamma)))\cap B_i\neq \phi$ for $i=1,2$. 
\end{enumerate}
\end{dfn}

The proof now reduces to two cases.

\emph{Case I:} Suppose first that for every $j$, $\{x_j^k\}_{k\in \mbb N}$ are not connecting. We will show that the decomposition of $F_k$ on $A$ extends to a decomposition on all of $\D$, which then allows us to independently normalize the boundary data to satisfy the three point condition of Theorem \ref{thm:courant}. 

To see this, choose $x_0 \in \D \setminus (A \cup \Sigma_0)$, and let $\gamma: [0,1]\rightarrow A$ be a path so that $\gamma(0) = x_0$ and $\gamma(1)=z \in A$ so that $\gamma([0,1]) \subset \D \setminus (A \cup \Sigma_0)$. Standard selection theory (see, for example the beginning chapters of \cite{revisit}) gives a unique selection 

\begin{equation*}
F_k\circ \gamma = \sum \llb f^i_k \rrb
\end{equation*}

\noindent Partition $F_k(x_0) = \sum \llb f^i_k(0) \rrb$ into two multiple-valued functions by:
\begin{equation*}
F_k(x_0) = \sum_{(z,f^i_k(1)) \in B_1} \llb f^i_k(0) \rrb + \sum_{(z,f^i_k(1)) \in B_2} \llb f^i_k(0) \rrb =:\llb F_k^1(x_0)\rrb +\llb F_k^2(x_0)\rrb
\end{equation*}

One checks that the fact that no element of $\Sigma_0$ is connecting implies that this choice is independent of path and therefore well-defined. Since each $F_k$ is continuous, this decomposition may be extended to the finite set $\Sigma_0$ by continuity. It readily follows that each $F_k^i$ is Sobolev. Further, the decomposition on $A$ gives (up to possible relabeling)

\begin{equation*}
F_k^1 |_{S^1} (s)= \sum_{\zeta=s^{k_1}} \llb f_k(\zeta) \rrb
\end{equation*}
\begin{equation*}
F_k^2 |_{S^1} (s)= \sum_{\zeta=s^{k_2}} \llb g_k(\zeta) \rrb
\end{equation*}

\noindent as in (1) at the beginning of the proof. Therefore, by picking two sequences of conformal disk automorphisms, $\{\sigma_k\}$ and $\{\psi_k\}$ so that $f_k \circ \sigma_k$ and $g_k \circ \psi_k$ satisfies a three point condition, we see that $\llb F_k^1 \circ \sigma_k \rrb+ \llb F_k^2 \circ \psi_k\rrb$ converges to an admissible map $\td F$ as $k\rightarrow \infty$--but since this normalization doesn't affect the energy, lower semicontinuity of the Dirichlet energy gives that $\td F$ is a minimizer, as desired.

\emph{Case II:} Suppose now that $\{x_1^k\}_k$ is a sequence of connecting branch points. We will prove that the boundary data $f_k$ and $g_k$ cannot degenerate in the limit.  Let $\alpha_1, \ \alpha_2,$ and $\alpha_3$ be three distinct points of $S^1$, and let $s_i^k = f_k(\alpha_i)$, $t_i^k = g_k(\alpha_i)$ for $i=1,2,3$. By compactness of $\Gamma_1$ and $\Gamma_2$, we may, upon extracting to a subsequence (which we again will not relabel) suppose that $s_i^k \rightarrow s_i$ and $t_i^k\rightarrow t_i$ for each $i$ as $k\rightarrow \infty$. 

By precomposing with a disk diffeomorphism, we may assume that the set $\{t_1,t_2,t_3\}$  consists of three distinct elements , and hence  that the boundary data $g_k$ converges uniformly to some $g$ by Theorem \ref{thm:courant}. We want to show that the set  $\{s_1,s_2,s_3\}$ also consists of three distinct elements, since then we may apply Theorem \ref{thm:courant} to obtain the result immediately.

We now assume towards a contradiction that $\{s_1, s_2, s_3\}$ are not distinct, say that $s_1 = s_3$. Under these assumptions, by the monotonicity of the $f_k$, and once again, possibly extracting a (still un-relabeled) subsequence  we may select a sequence $\beta^k$ with the following properties:

\begin{enumerate}
\item $\beta^k \rightarrow \beta \notin \{\alpha_1, \alpha_2\}$ as $k\rightarrow \infty$.
\item $f_k(\beta^k) = \td s_3\notin \{s_1, s_2\}$. 
\end{enumerate}

Note that $g$ must be non-constant on one of the arcs $(\beta, \alpha_1)$, or $(\beta, \alpha_2)$, where $(\beta,\alpha_1)$ is the arc that doesn't pass through $\alpha_2$ and vice versa.  We suppose that $g$ is non-constant on $(\beta, \alpha_1)$. Since for each $k$ there are only finitely many  branch points to avoid and their locations are constrained to balls of fixed radii, we may find a smoothly bounded, simply connected $U \subset \D$ satisfying the properties in the definition of connecting branch point, with the additional property that:

\begin{itemize}
\item $\partial U$ is a smooth curve whose intersection with $A$ has an arc $L$ so that $\text{dist}(L \cap (\beta, \alpha_1)) < \frac{\delta}{2}$, where $(\beta, \alpha_1)$ is the arc of $S^1$ connecting $\beta$ and $\alpha_1$ which does not pass through $\alpha_2$.
\end{itemize}

Using the estimate given by Inequality (\ref{eq:neccest}) in the proof of Lemma \ref{lemma:nbd}, we may further assert that if $r_f : U_f \rightarrow \Gamma_i$ is a Lipschitz retract and $\epsilon >0$ is given, we may choose $U$ so that $(\td s^3 \pm \epsilon , s_1 \pm \epsilon ) \subset r_f( F_k(L)\cap U_f)$, where the $\pm$'s are independent of one another. 

 Let $\phi : \D \rightarrow U$ be the conformal map guaranteed by the Riemann Mapping Theorem. Further, by Caratheadory's Theorem (see, for example Theorem 5.5 of \cite{Conway} and the preceeding sections), we see that $\phi$ extends to a homeomorphism  $\varphi: S^1 \rightarrow \partial U$. Then since $\{x_1^k\}$ is a sequence of connecting branch points, we may put

\begin{equation*}
F_k|_{\partial U}(z) = \sum_i \sum_{\zeta^{\ell_i} = \phi^{-1}(z) } \llb h_{k,i} (\zeta \rrb) .
\end{equation*}

\noindent For some integers $\ell_i$ with $\sum \ell_i = Q$ and some continuous functions $h_{k,i} : S^1 \rightarrow \R^n$.  Without loss of generality, we may assume that $h_{k,1}$ is a parameterizing map for $\Lambda$ as in the definition of connecting branch points. Set $h_k = h_{k,i}$. 

 Consider now the maps $\psi_k(z)= (z, h_k(z) )$ for $z\in \D$. Since $\partial U \Subset \D$, uniform convergence on compact sets of the maps $F_k$ guarantees that the curves $\psi_k(S^1)$ converge in the sense of Fr\' echet (see, for example, \cite[\S 4.2]{Hildmin} ) to a curve $\Psi$. Further, since $g_k$ converges uniformly to a non-constant function (by our choice of arc $(\beta, \alpha_1)$), by the comments in Remark \ref{rmk:addit}, the traces of $\psi_k$ are monotone and satisfy the three point condition of Theorem \ref{thm:courant} and thus $\text{Trace}(\psi_k)$ converges uniformly to some $\psi: S^1 \rightarrow \R^n$, which is monotonic onto the curve $\Psi$. Put $\psi(z) = (z, h(z))$. Then, for each $k$, there is a segment of the image curve of $h$ contained in $((\td s_3, s_1))_\delta$ (where this is the interval inside $\Gamma_1$), which follows from Remark \ref{rmk:addit} and so, by applying the retraction $r_f : (\Gamma_1)_\delta \rightarrow \Gamma_1$, we see that this property must persist in the limit, and therefore the $f_k$ restricted to the interval $(\beta, \alpha_1)$ cannot degenerate, and hence we may choose three points and apply Theorem \ref{thm:courant} to obtain convergence to a monotone map.

Finally, notice that we may (possibly upon extracting to another subsequence and utilyzing another diagonal argument) assume that the sequence of renormalized functions weakly converges to some $\td F \in W^{1,2}(\D, A_Q(\R^n))$. Further, Remark \ref{rmk:addit} applied to smaller and smaller neighborhoods paired with uniform convergence on compact sets and uniform convergence on $S^1$, one obtains using a standard $\frac{\epsilon}{3}$ argument that this subsequence converges uniformly to $\td F$ on $\D$, which implies that $\td F$ is admissible, and lower semi-continuity of the Dirichlet Energy then gives the result.

\end{proof}

\begin{rmk}[Generalization to $K >2$] 
\label{rmk:higherq} 
One notices that there is nothing unique for the case where $K=2$, there are simply fewer cases of the branching behavior. Indeed, one first checks that Claim \ref{claim:modification} holds for arbitrary $K$, and then, to guarantee the boundary convergence, one uses an identical argument as presented above, there are simply more cases to be concerned with since there are more possibilities for branching behavior. 
\end{rmk}

\section{Proof of the Regularity Theorem}
\label{section:topology}
In this section we  prove Theorem \ref{thm:conformal} via  methods gleaned from the classical situation for Douglas Minimizers. The proof involves using topological methods and complex analysis to produce a single-valued map from a planar domain with the same minimization properties and boundary data. It relies heavily on the classical results regarding the Douglas problem--see Chapter 8 of \cite{Hildmin}. 

Before we begin, we need a few results from complex analysis. We start by stating a weakened version of a theorem of Bieberbach \cite{bieber}--see also \cite{complex}.

\begin{thm}
\label{thm:bieber}
Let $\Omega\subset \C$ be a bounded domain in the plane bounded by $n$ non-intersecting Jordan curves $\gamma_1,\gamma_2,...,\gamma_n$, and, for each $i$, choose a point $b_i \in \gamma_i$. Then there exists a proper holomorphic mapping $f: \Omega \rightarrow \D$ which is an $n$-to-one branched covering map. Further, since $f$ is proper, $f$ extends continuously to the boundary of $\Omega$, and this extension (which we will also denote by $f$) maps each boundary curve monotonically onto the unit circle and has $f(b_i) =1$ for all $i$. 
\end{thm}

The above Theorem leads, in a fairly straightforward way, to the following. The main idea is to translate the following Theorem into a problem on the upper half-plane, and then add functions produced via the above with clever selections of points--for full details, see \cite[\S 5]{complex}.

\begin{thm}
\label{thm:genahlfors}
Suppose that $\Omega \subset \C$ is as in Theorem \ref{thm:bieber} and $d_1,...,k=d_n$ are positive integers. Then there exists a proper holomorphic map $f: \Omega \rightarrow \D$ such that:

\begin{enumerate}
\item $f$ is a $d_1+d_2+...+d_n$-to-one branched cover of $\D$
\item $f$ admits a continuous extension to $\overline \Omega$ (which we will also denote by $f$) so that, for each $i$, $f|_{\gamma_i}$ has degree $k_i$.
\end{enumerate}

\end{thm}

\begin{proof}[Proof of Theorem \ref{thm:conformal}]   Let $F$, $\Gamma_1,...,\Gamma_K$ and $k_1,...,k_K$ be as in Definition \ref{dfn:wrapped}. Consider the set $\Gamma(F) = \{ (x,y) \ | \ x\in \D \ y\in\spt(F(x))\}$, the \emph{graph} of $F$ (note that $\Gamma(F)$ is also the multiple-valued image of the map $GF$ from the above sections). 

Notice that, since the boundary data admits a wrapped solution, $F$ can have no branch points on the boundary $S^1$, and thus by Remark \ref{rmk:addit} $F$ has no branch points on some annulus $A$, and hence only finitely many branch points in $\D$. 

Next let $\pi_1$ denotes the projection of $\R^2 \cross \R^n$ onto the $\R^2$ coordinates, and $\pi_2$ denote the projection onto $\R^n$ coordinates. Comments in the above paragraph show that $\pi_1 : \Gamma(F) \rightarrow \D$  is a branched cover.

Further, one readily checks that $\Gamma(F)$ is a Riemann surface with boundary for which any loop is either null-homotopic or homotopic to a concatenation of boundary curves--and so $\Gamma(F)$ is therefore topologically a punctured Riemann Sphere. By Koebe's General Uniformization Theorem (see, for example, \cite{foundit}) that there is a planar domain $\Omega$ with $K$ boundary components and a conformal diffeomorphism $\varphi : \Omega \rightarrow \Gamma(F)$.

By conformal invariance of energy and construction, $\kappa :=\pi_2 \circ \varphi : \Omega \rightarrow \R^n$ has the same energy as (the multiple-valued function) $F$. 

Our goal will be to show that that $\kappa$ is conformal, as this will imply (2) of Theorem \ref{thm:conformal}. (2) paired with the energy/area inequality will then imply (1).

Supposing towards a contradiction, suppose that $\kappa$ is not conformal. Following \cite[\S 8.2 and \S 4.5]{Hildmin}, we find a domain $\td \Omega$ and a diffeomorphism $\sigma : \td \Omega \rightarrow \Omega$ so that $\Dir(\kappa \circ \sigma) < \Dir(\kappa)$. Further, we may assume that $\sigma$ is arbitrarily close to the identity in the uniform topology, and thus that $\td \Omega$ has the same number of boundary components as $\Omega$.

Our next goal is to use $\td \Omega$ and $\sigma$ to produce a competitor for $F$ with less energy, which will contradict $F$'s minimizing property. 

Let $\gamma_1,...,\gamma_k$ be the boundary curves of $\td \Omega$ ordered so that $\kappa \circ \sigma (\gamma_i) \subset \Gamma_i$. Then, apply Theorem \ref{thm:genahlfors} to $\td \Omega$ with $d_i = k_i$ to obtain a $Q$-to-one branched cover $\Phi :\td  \Omega \rightarrow \D$ which extends to a map on the closure of $\td \Omega$ for which the maps $\Phi|_{\gamma_i}$ are monotonic and of degree $k_i$ onto the circle $S^1$. 

We then define a multiple-valued map $G : \D \rightarrow A_Q$ as follows: for $z$ a regular point of the branched cover $\Phi$,

\begin{equation}
G(z) = \ds \sum_{y\in \Phi^{-1}(z) } \llb \kappa \circ \sigma (y) \rrb
\end{equation}

\noindent and extended by continuity to the finitely many branch points of $\Phi$. One readily checks that the energy of $G$ is (since $\Phi$ is holomorphic and extends to a monotonic map on the boundary, and hence conformal except at finitely many interior points) equal to the energy of $\kappa \circ \sigma$, a contradiction, since $G$ is also admissible. Therefore $\kappa$ is conformal, proving (2) and hence (1).

\noindent 

To prove (3), we follow an argument similar to the one presented in \cite[Theorem 3, \S 4.5]{Hildmin}. If one of the $f_i's$ were not a homeomorphism, since it is monotonic it must map some arc $\gamma$ along $S^1$ to a constant $\alpha \in \R^n$. Let $x_0^{k_i}$ be some point on the interior of this arc. We then find a harmonic selection on some small ball $B_r(x_0) \cap \D$ by Lemma \ref{lemma:nbd}. By (1), and since this selection is unique (up to reordering), it must be conformal. Let $F_i$ be the function on $B_r(x_0)\cap \D$ which maps $B_r(x_0) \cap \gamma$ to a constant. By Schwarz reflection and conformality, we obtain that $F_i \equiv \alpha$ and hence, since $\D \setminus \sigma(F)$ is path connected (where $\sigma(F)$ is the set of branch points of $F$), that $\alpha\in \spt(F(z))$ for all $z\in \D$, an obvious impossibility given the separation assumptions on the curves $\Gamma_i$. 

\end{proof}

\section{A Class of Examples}

In this section we produce an example of two curves $\Gamma_1$ and $\Gamma_2$ for which the Plateau solution produced in the above section does not allow for a global selection so that there is necessarily a branch point and the image is connected. Further, we give a method for producing a large number of such examples.

Our last proposition is the following:

\begin{prop}
\label{prop:classex}
Suppose that $V \subset \C \cross \C^n$ is a holomorphic variety for which the projection $\pi (z_1,...,z_{n+1} ) = (z_1)$ is a $Q$ to $1$ cover of the disk $\D$ except for possibly finitely many points on the interior of $\D$, and for which $S^1 \cross \C^n$ intersects $V$ in $Q$ disjoint curves. Then the multiple-valued function $ z\mapsto \pi^{-1} (z) \cap V$ is the Plateau solution for the given curves.
\end{prop}

The proof of the proposition relies on the following elementary lemma.

\begin{lemma}
\label{lemma:mass}
Suppose $\Omega \subset \R^2$ is compact, and that $F: \Omega \rightarrow \R^n$ is injective and Lipschitz. Then, if $\td F (x,y) = (x,y,F(x,y)) \in \R^2 \cross \R^n$,and $\llb\Omega\rrb$ is the standardly-oriented  two-dimensional current in $\R^2$, we have:

\begin{equation*}
\mass (F_\sharp (\llb \Omega\rrb) ) +|\Omega| \leq  \mass(\td F_\sharp(\llb\Omega)\rrb )
\end{equation*}
\end{lemma}

\begin{proof}
One readily checks that, if $a$ and $b$ are non-negative so that:

\[ DF^T \circ DF =\left( \begin{array}{cc}
a & b  \\
c &d  \\
\end{array} \right)\] 

\noindent then,

\[ D\td F^T \circ D\td F =\left( \begin{array}{cc}
1+a & b  \\
c &1+d  \\
\end{array} \right).\]

\noindent Therefore,

\begin{equation*}
1 + \sqrt{ \det (DF^T \circ DF)} \leq \sqrt{\det ( D\td F^T \circ D\td F) } 
\end{equation*}

\noindent which combined with the area formula implies the result. 

\end{proof}

To prove the proposition, we recall two results from \cite{higher}, see Lemma 1.8 and Proposition 2.2 and the definitions preceeding them for more details: 

\begin{thm}
\label{thm:higherests}

Let $f \in W^{1,2}(\Omega, A_Q(\R^n)$ be so that $\mass (T_{f,\Omega}) < \infty$. Then:

\begin{equation*}
\mass (T_{\lambda f, \Omega}) = Q|\Omega| + \frac{\lambda ^2}{2} \dir(f,\Omega) + o(\lambda^2) \ \mbox{as $\lambda \rightarrow 0$}
\end{equation*}

\noindent Further, if $F$ is the associated $Q$-valued function of a holomorphic variety over $\Omega \subset \C$, then:

\begin{equation*}
\mass (\llb V \rrb \restr \Omega \cross \C^n ) = Q |\Omega| + \frac{\dir(F,\omega)}{2}
\end{equation*}

\end{thm}

\begin{proof}[Proof of Proposition \ref{prop:classex}]

Note that $V_\lambda = \{ (z,\lambda w) \ | \ (z,w) \in V\}$is also a holomorphic variety, and, if $F_\lambda$ is the corresponding $Q$-valued function, Theorem \ref{thm:higherests} yields:

\begin{equation*}
\mass (\llb V_\lambda \rrb \restr \D \cross \C^n ) = Q |\D| + \frac{\dir(F_\lambda,\D)}{2} = Q |\D| + \frac{\lambda^2}{2} \dir(F,\D)
\end{equation*}

Let $G : \D \rightarrow \C \cross \C^n$ be any admissible map for $\Gamma_1,...,\Gamma_Q$. Then, note that if $P : \C \cross \C \cross \C^n$ is the projection onto the last components, $\pi (z, w, v) = (w,v)$ for $z,w\in \C$ and $v\in \C^n$ and $\lambda >0 $ is given,

\begin{equation*}
\mass (\llb V_\lambda \rrb \restr \D \cross \C^n ) \leq \mass ( P_\sharp (T_{\lambda G, \D}) )
\end{equation*}

\noindent since complex varieties are absolutely mass minimizing and $G$ is admissible. 

However, using the above inequalities, this amounts to:

\begin{equation*}
 Q |\D| + \frac{\lambda^2}{2} \dir(F,\D) \leq \mass ( P_\sharp (T_{\lambda G, \D}) ).
\end{equation*}

\noindent However, by the definition of the current $T_{\lambda G, \D}$ and Lemma \ref{lemma:mass}, we see that the above gives:

 \begin{equation*}
2 Q |\D| + \frac{\lambda^2 \dir(F,\D)}{2} \leq \mass ( P_\sharp (T_{\lambda G, \D}) )  + Q |\D| \leq   \mass ( T_{\lambda G, \D} ) = Q|\D| + \frac{\lambda ^2}{2} \dir(f,\D) + o(\lambda^2) 
\end{equation*}

\noindent and therefore, 

\begin{equation*}
Q |\D| + \frac{\lambda^2}{2} \dir(F,\D)  \leq  \frac{\lambda ^2}{2} \dir(G,\D) + o(\lambda^2) 
\end{equation*}

\noindent Finally, notice that $\lambda F$ allows for a  local conformal selection, and, if $\td F$ denotes the graph map as in Lemma \ref{lemma:mass}, that $1 + \frac{\lambda ^2}{2} |DF|^2  = \sqrt{\det (D\td F^T \circ \td F )}$, so the area-energy equality for conformal maps, the above gives:

\begin{equation*}
\frac{\lambda^2}{2} \dir(\td F,\D)  \leq  \frac{\lambda ^2}{2} \dir(G,\D) + o(\lambda^2) 
\end{equation*}

\noindent which implies $\dir(\td F,\D)  \leq  \dir(G,\D) $, proving the proposition.

\end{proof}

For a specific example, consider the multiple-valued map $F$ defined by $z \mapsto (z,\sqrt{z^2 - \frac{1}{4}})$ for $z\in \C$.  It is clear that $F$ has two connecting branch points located at the points $z=\pm \frac{1}{2}$.

Note that $F|_{S^1}$ has a selection given by:

\begin{itemize}
\item $f_1(\theta) = (\exp (i \theta),\sqrt[4]{ (1+\frac{1}{16} - \frac{\cos(2\theta)}{2}) }\exp(i \cdot\frac{ \atantwo (\sin(2\theta), \cos(2\theta)-\frac{1}{4})}{2} ))$
\item $f_2(\theta)= (\exp (i \theta),  \sqrt[4]{ (1+\frac{1}{16} - \frac{\cos(2\theta)}{2}) }\exp(i\cdot (\frac{  \atantwo (\sin(2\theta), \cos(2\theta)-\frac{1}{4})}{2} )+\pi))$
\end{itemize}

\noindent Let $f_i = (f^1_i,f^2_i,f^3_i)$. Then, even though the above selection is discontinuous it may be used to show that the boundary has two connected components, as is suggested by the following images of the real and imaginary parts of the functions $f_i^3$.

\begin{center}
\includegraphics[scale=.5]{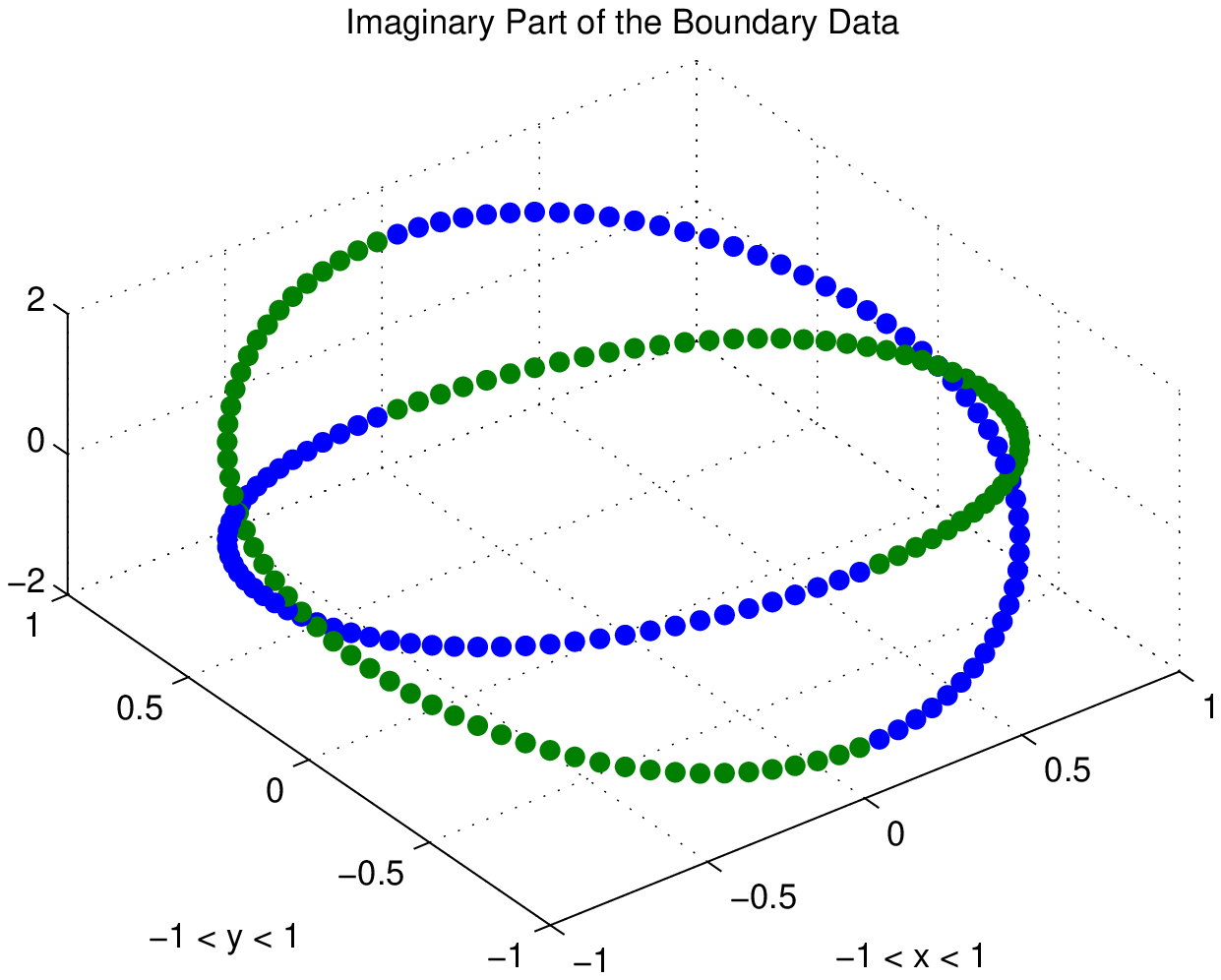} \vspace{5mm}\includegraphics[scale=.5]{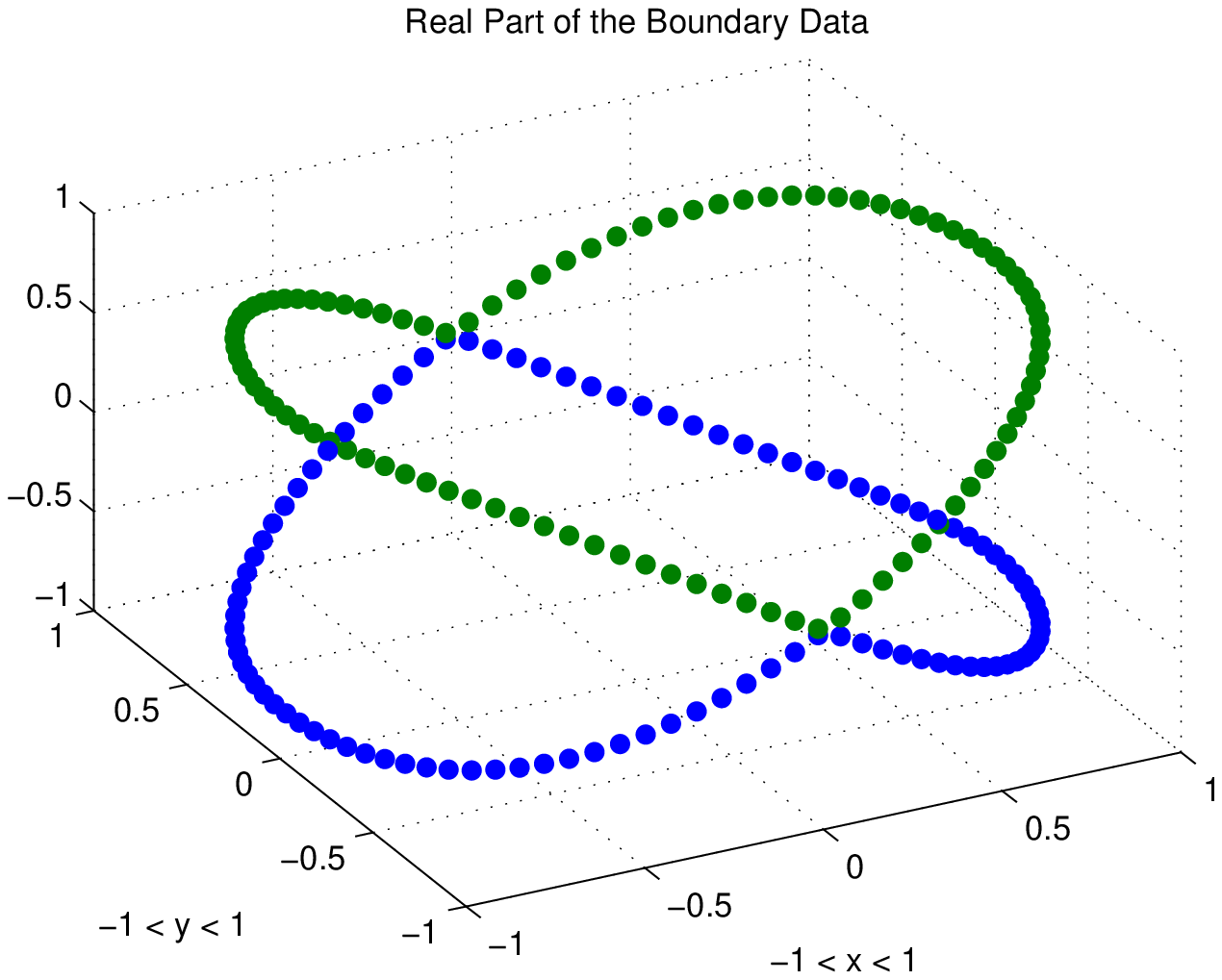} \vspace{5mm}
\end{center}

\noindent   One may check that $\Re(f^3_1(\theta)) = \Re(f^3_2(\theta) ) $ if and only if $\theta \in \{ \pm \frac{\pi}{2}\}$ and $\Im(f^3_1(\theta)) = \Im(f^3_2(\theta))$ if and only if $\theta \in \{0,\pi\}$.

Further, these images suggest that the following will give a continuous selection:

 \begin{displaymath}
   g_1(\theta) = \left\{
     \begin{array}{llr}
      f_1(\theta)& : & \theta \in [\frac{-\pi}{2},\frac{\pi}{2}]\\
      f_2(\theta) & : & \theta \in [\frac{\pi}{2},\frac{3\pi}{2}]\\
     \end{array}
   \right.
\end{displaymath}

 \begin{displaymath}
   g_2(\theta) = \left\{
     \begin{array}{llr}
      f_2(\theta)& : & \theta \in [\frac{-\pi}{2},\frac{\pi}{2}]\\
      f_1(\theta) & : & \theta \in [\frac{\pi}{2},\frac{3\pi}{2}]\\
     \end{array}
   \right.
\end{displaymath}

\noindent Which one checks provides a continuous selection of the boundary data. Further, by (locally) holomorphically defining $z \mapsto \sqrt{z \pm \frac{1}{2}}$ we see that, for sufficiently small $\rho > 0$, $F|_{\partial B_\rho (\pm \frac{1}{2}) }$ is a connecting branch point, since then it has  structure $\varphi \circ \sqrt{z\mp\frac{1}{2}} $, for $\varphi$ holomorphic.  Therefore, Proposition \ref{prop:classex} guarantees that $F$ is the Plateau solution, and thus that the Plateau solution is branched.

Branch points could also occur in $\mathbb R^3$ for essentially topological reasons.  For $0<\epsilon<1$, consider on the circle $S^1$, a single branch of the 2-valued function 
$\phi(z)=(\sqrt{z-\epsilon}, {\mathcal Re}\sqrt{z-\epsilon})$. Note that the image, {\it under this 2 valued map} is a single smoothly embedded circle in $\mathbb R^2\times\mathbb R$ whose projection onto $\mathbb R^2$ circulates the origin twice.  Then $\phi$ does not admit a continuous $2$-valued extension $F:\mathbb D\to A_2(\mathbb R^2\times\mathbb R)$ which is a locally smooth embedding without branch points. In fact otherwise one could start at the origin with two values and smoothly extend to obtain a selection on the whole disk, the image of whose boundary values would, unlike $\phi$, be two embedded curves. In particular, any 2-valued Plateau solution in $\R^3$ starting with this ``doubly-wrapped'' 2-valued boundary $\phi$ must have branch point.  By the same argument we see that:

{\it Any solution of the Plateau problem in Theorem 1.2 starting with boundary parameterization as in Theorem 1.1 with multiple wrapping (i.e. some $k_i > 1$) , whose boundary curves satisfy condition \ref{dfn:wrapped}  must have a branch point. In case $N=3$, the image then must curves of self-intersection. } 

It is the subject of current work to determine whether or not one can verify the conditions of Definition \ref{dfn:wrapped} in $\R^3$. 

However, with $Q=2$ and $k_1=k_2=1$, for two closed curves lying in different planes in $\R^3$, one can, using similar arguments to those in \cite[\S 8.8]{Hildmin} to guarantee that our Plateau solutions do in fact admit branch points in three dimensional space. Further, constancy along a boundary curve \emph{can} have elaborate branching behavior, as illustrated in the example below--although this example does differ slightly from the type of degeneracy that would occur if given boundary data did not admit a wrapped solution. 

\section{An Example of a Degenerate Behavior}
\label{section:degen}

In this section we provide a counterexample to the following question:

\begin{question}
\label{question:main} Let $Q\geq 2$ and suppose that $F: \D \rightarrow A_Q(\R^n)$ is Dirichlet minimizing, and $F|_{S^1} = \sum_{i=1}^{Q-1} \llb f_i \rrb+\llb 0 \rrb$, for monotone, continuous maps $f_i: S^1 \rightarrow \Gamma_i$, with $\Gamma_i$ closed Jordan curves with $0\notin \Gamma_i$. Then, is  $0 \in \spt(F(x))$ for all $x\in \D$? 
\end{question}

Note that the answer is \emph{yes} for $Q=1$, by the maximum principle for harmonic functions. Similarly,  for $n=1$, the answer to Question \ref{question:main} is again \emph{yes}, since one may find a continuous (and therefore Sobolev, since $F$ will be affinely approximatable almost everywhere, with $|DF|$ square integrable) selection for $F$, one of whose traces is zero on the boundary.

 These facts leads to the question for higher $Q$--and is an \emph{a priori} type of degeneracy that can occur for certain types of sequences that could have occurred in our existence proof. The answer to this question also illustrates just how weak the maximum principal (as in Section 3.2 of \cite{revisit}) for multiple-valued functions actually is.

Our counter example is in the case $Q=2$ and $n=2$, and our proof is unique to this case. However, simple examples show that this example may be extended to all $Q \geq 2$, provided $n > 1$. 

The key is to look instead for zero average multiple-valued functions on $D$ that admit a selection on $S^1$. The following notation will be useful in proving an equivalence between the two concepts.

\begin{notation}
\label{notation:sums}
For $f: \D\rightarrow \R^n$ and $H \in W^{1,2} (\D ,  A_2(\R^n))$, we denote by $H+f$ the function $\llb h_1 + f\rrb + \llb h_2 + f \rrb$, where $h_1$ and $h_2$ is a measurable selection for $H$. Note that $H+f$ is independent of the selection chosen.
\end{notation}

Next we recall a result of \cite{higher}:

\begin{thm}
\label{thm:mins}
Let $V\subset \C^\mu \cross \C^\nu$ be an irreducible holomorphic variety which is a $Q$-to-one cover of the ball $B_2\subset \C^\mu$ under the orthogonal projection. Then there is a Dirichlet minimizing $f\in W^{1,2}(B_1,  A_Q(\mbb R^{2\nu} ) )$ so that graph$(f)=V \cap (B_1 \cross C^\mu)$. 
\end{thm}

Consider the variety $V=\{ (z,w) \ | \ z^2-\frac 1 4  = w^2 \}$. It is easy to see that $p(z,w) = w^2 - z^2 - \tfrac 1 4$ is irreducible. Notice that for $z\neq \pm \frac 1 2$, $V$ is $2:1$ cover under the orthogonal projection. Therefore,  Theorem \ref{thm:mins}\footnote{In the proof of Theorem \ref{thm:mins}, \cite{higher} actually only requires (using standard notation for currents in Euclidean space) that $\pi_\sharp (\llb V  \rrb ) = 2 \llb B_2 \rrb$} gives that the multiple-valued function $z\mapsto \sqrt{ z^2-\frac 1 4}$ is Dirichlet Minimizing over the disk $\D$.

Further, as shown in Section 6, $F$ admits a continuous selection on the boundary $S^1$ of $\D$, since this result follows from the fact that the graph map $z\mapsto (z, \sqrt{z^2-1/4})$ maps $S^1$ to two components. We also notice that $F$ is symmetric--i.e. that if $f_1$ and $f_2$ represent a measurable selection for $F$, then, since $f_1+f_2$ is harmonic and zero on the boundary, we must have $f_1 + f_2 =0$. 

Therefore, let $F|_{S^1} = \llb f \rrb + \llb - f \rrb$. With a slight abuse of notation, let $f$ also denote the harmonic extension of $f$ to the disk $\D$. We finally check that, since $F$ is Dirichlet minimizing, $\td F = F+f$ defined as in Notation \ref{notation:sums} is also Dirichlet minimizing for its boundary data.

This follows immediately from the following proposition.

\begin{prop}
Let $H \in W^{1,2} (\D ,  A_2(\R^n))$ be Dirichlet minimizing and suppose that $H|_{S^1} = \llb H_1 \rrb + \llb H_2\rrb$ for $H_1, \ H_2 : S^1 \rightarrow \R^n$ continuous. Let $h$ denote the harmonic extension of $H_1$ to $\D$. Then:

\begin{enumerate}
\item if $H_1 = -H_2$, then $\Dir(H + h, \D) =\Dir(H, D) +2\Dir(h, \D)$
\item if $H_2 = 0$, then $\Dir(H - \frac{h}{2}, \D) =\Dir(H, D) -2\Dir( \frac{h}{2}, \D)$
\end{enumerate}
\end{prop}

\begin{proof}[Proof of Proposition]

Denote for two $m\cross n$ matricies $(a_{ij})$ and $(b_{ij})$ the matrix inner product:

\begin{equation*}
( (a_{ij} ): (b_{ij}) ) = \sum_{i,j} a_{ij}\overline{b_{ij}}.
\end{equation*} 

\noindent Note that the norm induced by this inner product is the Hilbert-Schmidt norm.

Let $h_1$ and $h_2$ denote a measurable selection for $G$. In case (1), properties of the inner product above imply that at any point where $G$ is affinely approximatable (which constitutes a full measure set of $\D$) we have:

\begin{equation*}
||Dh_1 + Dh || ^2 + ||Dh_2 + D h||^2 =  ||DH||^2 + 2||Dh||^2 + 2(Dh_1 + Dh_2 : Dh) .
\end{equation*}

\noindent But, $ 2(Dh_1 + Dh_2 : Df) =0$ since $h_1 + h_2 = 0$, and this gives the result.

For case (2):

\begin{equation*}
||Dh_1 - \frac{Dh} 2  || ^2 + ||Dh_2 -\frac{Dh} 2||^2 =  ||DH||^2 + 2||\frac{Dh} 2||^2 - 2(Dh_1 + Dh_2 :\frac{Dh} 2) .
\end{equation*}

\noindent In this case, though, since $h_1+h_2$ is harmonic and has the same boundary data as $h$ it must be the case that  $h_1+h_2 = h$, so:

\begin{equation*}
 2(Dh_1 + Dh_2 :\frac{Dh} 2) = 2(Dh,\frac{Dh}2) = 4 ||\frac{Dh}{2}||^2.
 \end{equation*}
 
 \noindent Which implies (2).

\end{proof}

Now we prove $\td F$ is Dirichlet minimizing. (1) of the above proposition guarantees:

\begin{equation*}
\Dir(\td F, \D) = \Dir (F+f, \D) = \Dir(F,\D) + 2 \Dir(f, \D)
\end{equation*}

Let $G \in W^{1,2} (\D ,  A_2(\R^n))$ be Dirichlet minimizing and suppose that $G|_{S^1} =\td F_{S^1}=\llb 2f\rrb + \llb 0\rrb$. Then case (2) of the proposition implies:

\begin{equation*}
 \Dir (G-f, \D) = \Dir(G,\D) - 2 \Dir(f, \D).
\end{equation*}

\noindent So, if

\begin{equation*}
 \Dir(G,\D) < \Dir(\td F, \D)
 \end{equation*}
 
 \noindent then:

\begin{equation*}
\Dir(G,\D) < \Dir(F, \D) + 2\Dir(f,\D)
\end{equation*}

\noindent which implies:

\begin{equation*}
\Dir(G-f,\D)=\Dir(G,\D)- 2\Dir(f,\D) < \Dir(F, \D) .
\end{equation*}

\noindent However, this last inequality is impossible since $F$ is minimizing, so $\td F$ is Dirichlet minimizing.

 However, since $F$ admits no global selection (due to its connecting branch points), $\td F$ also admits no such selection, and therefore $0\notin \spt(\td F(x))$ for some $x\in \D$, so $\td F$ is a counterexample to Question \ref{question:main}.

\bibliographystyle{plain} \bibliography{reference}

\end{document}